\documentclass[a4paper,11pt]{amsart}
\usepackage{amssymb}
\usepackage{pifont} 
\usepackage{bbding}
\usepackage{wasysym}
\usepackage{ifthen}
\usepackage{cite}
 \usepackage[dvips]{graphicx}
\nonstopmode \numberwithin{equation}{section}
\usepackage{amssymb}
\usepackage{ifthen}
\usepackage{graphicx}
\usepackage{amsmath}
\usepackage[T1]{fontenc} 
\usepackage[utf8]{inputenc}
\usepackage[usenames,dvipsnames]{color}
\usepackage{color}
\usepackage[english]{babel}
\usepackage{fancyhdr}
\usepackage{fancybox}
\usepackage{tikz}

\nonstopmode \numberwithin{equation}{section}
\theoremstyle{plain}

\newtheorem{conj}{Conjecture}

\theoremstyle{definition}
\newtheorem{defi}{Definition}[section]

\newtheorem{thm}{Theorem}[section]
\newtheorem{prob}{Problem}[section]
\newtheorem{cor}{Corollary}[section]

\newtheorem{prop}{Proposition}[section]
\newtheorem{rem}{Remark}[section]
\newtheorem{lem}{Lemma}[section]

\newcounter{minutes}
\divide\time by 60
\newcounter{hours}
\multiply\time by 60
\addtocounter{minutes}{-\time}

\newcounter {own}
\def\theown {\thesection       .\arabic{own}}

\newenvironment{pf}[1][]{%
 \vskip 3mm
 \noindent
 \ifthenelse{\equal{#1}{}}%
  {{\slshape Proof. }}%
  {{\slshape #1.} }%
 }%
{\qed\bigskip}

\newcounter{alphabet}

\def\be{\begin{equation}}
\def\ee{\end{equation}}

\newcommand{\bee}{\begin{enumerate}}
\newcommand{\eee}{\end{enumerate}}

\newcommand{\blem}{\begin{lem}}
\newcommand{\elem}{\end{lem}}
\newcommand{\bthm}{\begin{thm}}
\newcommand{\ethm}{\end{thm}}
\newcommand{\bcor}{\begin{cor}}
\newcommand{\ecor}{\end{cor}}
\newcommand{\beg}{\begin{examp}}
\newcommand{\eeg}{\end{examp}}
\newcommand{\begs}{\begin{examples}}
\newcommand{\eegs}{\end{examples}}

\newcommand{\bdefn}{\begin{defn}}
\newcommand{\edefn}{\end{defn}}

\newcommand{\bprob}{\begin{prob}}
\newcommand{\eprob}{\end{prob}}
\newcommand{\bei}{\begin{itemize}}
\newcommand{\eei}{\end{itemize}}

\newcommand{\bcon}{\begin{conj}}
\newcommand{\econ}{\end{conj}}
\newcommand{\bcons}{\begin{conjs}}
\newcommand{\econs}{\end{conjs}}
\newcommand{\bprop}{\begin{prop}}
\newcommand{\eprop}{\end{prop}}
\newcommand{\br}{\begin{rem}}
\newcommand{\er}{\end{rem}}
\newcommand{\brs}{\begin{rems}}
\newcommand{\ers}{\end{rems}}
\newcommand{\bo}{\begin{obser}}
\newcommand{\eo}{\end{obser}}
\newcommand{\bos}{\begin{obsers}}
\newcommand{\eos}{\end{obsers}}
\newcommand{\bpf}{\begin{pf}}
\newcommand{\epf}{\end{pf}}
\newcommand{\ba}{\begin{array}}
\newcommand{\ea}{\end{array}}
\newcommand{\beq}{\begin{eqnarray}}
\newcommand{\beqq}{\begin{eqnarray*}}
\newcommand{\eeq}{\end{eqnarray}}
\newcommand{\eeqq}{\end{eqnarray*}}

\begin{document}

\title{Sharp inequalities for Logarithmic Coefficients for Certain Classes of Univalent Functions}

\author{Sanju Mandal}
\address{Sanju Mandal, Department of Mathematics, Jadavpur University, Kolkata-700032, West Bengal, India.}
\email{sanju.math.rs@gmail.com}

\author{Molla Basir Ahamed}
\address{Molla Basir Ahamed, Department of Mathematics, Jadavpur University, Kolkata-700032, West Bengal, India.}
\email{mbahamed.math@jadavpuruniversity.in}

\author{Paweł Zaprawa}
\address{Paweł Zaprawa, Department of Mathematics, Lublin University of Technology, Nadbystrzycka 38D, Lublin, 20-618, Poland.}
\email{p.zaprawa@pollub.pl}

\subjclass[{AMS} Subject Classification:]{Primary 30C50, Secondary 30C45}

\def\thefootnote{}
\footnotetext{ {\tiny File:~\jobname.tex,
printed: \number\year-\number\month-\number\day,
          \thehours.\ifnum\theminutes<10{0}\fi\theminutes }
} \makeatletter\def\thefootnote{\@arabic\c@footnote}\makeatother

\begin{abstract} 
Let $\mathcal{S}$ denote the class of functions $f(z) = z + \sum_{n=2}^{\infty} a_n z^n$ that are analytic and univalent in the open unit disk $\mathbb{D} = \{z \in \mathbb{C} : |z| < 1\}$. In this paper, we determine the sharp bounds of the Toeplitz determinants whose entries are the logarithmic coefficients of $f \in \mathcal{S}$. Furthermore, we investigate the corresponding Toeplitz determinants for the logarithmic coefficients of the associated inverse functions. These sharp bounds are established for functions belonging to several well-known subclasses of $\mathcal{S}$, namely, the classes $\mathcal{S}^*(\alpha)$ of starlike functions of order $\alpha$, $\mathcal{C}(\alpha)$ of convex functions of order $\alpha$, $\mathcal{S}^*_{\alpha}$ and $\mathcal{C}_{\alpha}$ of strongly starlike and strongly convex functions of order $\alpha$, and $\mathcal{R}(\alpha)$ of functions with bounded turning. As special cases of our main results, we obtain the exact bounds of these determinants for the classical classes of starlike, convex, and bounded turning functions.
\end{abstract}

\maketitle
\pagestyle{myheadings}
\markboth{S. Mandal, M. B. Ahamed, and P. Zaprawa}{On logarithmic coefficients and inverse functions}

\section{\bf Introduction}
The coefficient problem constitutes a central theme in geometric function theory, with a primary focus on establishing sharp bounds for various functional inequalities. Let $\mathcal{H}$ denote the class of functions $f$ that are holomorphic (analytic) in the open unit disk $\mathbb{D} = \{z \in \mathbb{C} : |z| < 1\}$. Let $\mathcal{A}$ be the subclass of $\mathcal{H}$ consisting of functions normalized by the conditions $f(0) = 0$ and $f^{\prime}(0) = 1$. Thus, any function $f \in \mathcal{A}$ possesses a Taylor-Maclaurin series expansion of the form\begin{equation}\label{eq-1.1}f(z) = z + \sum_{n=2}^{\infty} a_n z^n, \quad z \in \mathbb{D}.\end{equation}We denote by $\mathcal{S}$ the subclass of $\mathcal{A}$ consisting of functions that are univalent in $\mathbb{D}$. For a comprehensive treatment and classical results on univalent functions, we refer the reader to the monographs of Duren \cite{Duren-1983-NY} and Goodman \cite{Goodman-1983}.

Univalent functions play a pivotal role in geometric analysis, finding extensive applications in conformal mappings, potential theory, and the geometry of Riemann surfaces. The structural study of these functions and their respective coefficients not only deepens the theoretical framework of complex analysis but also provides essential mathematical tools utilized across various disciplines in science and engineering. \vspace{2mm}

While numerous investigations in the literature have focused on establishing sharp bounds for Hankel determinants involving the logarithmic coefficients of various univalent function classes, the estimation of the corresponding Toeplitz determinants remains an open and fertile area of research. Driven by this motivation, the primary objective of the present paper is to determine the sharp bounds for Toeplitz determinants whose entries are the logarithmic coefficients of functions $f$ belonging to the subclasses $\mathcal{S}^*(\alpha)$, $\mathcal{C}(\alpha)$, $\mathcal{S}^*_{\alpha}$, $\mathcal{C}_{\alpha}$, and $\mathcal{R}(\alpha)$, respectively.

To contextualize our findings within recent developments and facilitate the presentation of our main results, we first recall the method of differential subordination. This powerful geometric technique serves as an indispensable tool in function theory, offering a highly effective and precise framework for resolving complex coefficient and structural problems.\vspace{2mm}
\begin{defi}
	Let $f$ and $g$ be analytic functions in the open unit disk $\mathbb{D} = \{z \in \mathbb{C} : |z| < 1\}$. The function $f$ is said to be \textit{subordinate} to $g$, denoted by $f \prec g$ (or $f(z) \prec g(z)$), if there exists a Schwarz function $\omega$, analytic in $\mathbb{D}$ with $\omega(0) = 0$ and $|\omega(z)| < 1$ ($z \in \mathbb{D}$), such that 
	\begin{equation*}
		f(z) = g(\omega(z)), \quad z \in \mathbb{D}.
	\end{equation*}
	In particular, if the function $g$ is univalent in $\mathbb{D}$, then the subordination $f \prec g$ is equivalent to the simultaneous conditions:
	\begin{enumerate}
		\item[(i)] $f(0) = g(0)$,
		\item[(ii)] $f(\mathbb{D}) \subseteq g(\mathbb{D})$.
	\end{enumerate}
\end{defi}
\begin{defi}
A domain $\Omega\subseteq \mathbb{C}$ is said to be starlike with respect to a point $z_0\in\Omega$ if the line segment joining $z_0$ to any point in $\Omega$ lies entirely in $\Omega$. If $z_0$ is the origin, then we say that $\Omega$ is a starlike domain. A function $f\in\mathcal{A}$ is said to be starlike if $f(\mathbb{D})$ is a starlike domain. We denote by $\mathcal{S}^{*}$ the class of starlike functions $f$ in $\mathcal{S}$. It is well known that a function $f\in\mathcal{A}$ is in $\mathcal{S}^{*}$ if and only if
\begin{align}\label{eq-1.2}
	{\rm Re}\left(\frac{zf^{\prime}(z)}{f(z)}\right)>0 \hspace{0.5cm} \mbox{for}\;\; z\in\mathbb{D}.
\end{align}
\end{defi}

\begin{defi}
A domain $\Omega\subseteq \mathbb{C}$ is said to be convex if the line segment joining any two points of $\Omega$ lies entirely in $\Omega$. A function $f\in\mathcal{A}$ is said to be convex if $f(\mathbb{D})$ is a convex domain. We denote by $\mathcal{C}$ the class of convex functions in $\mathcal{S}$. A function $f\in\mathcal{A}$ is in $\mathcal{C}$ if and only if
\begin{align}\label{eq-1.3}
	{\rm Re}\left(1+ \frac{zf^{\prime\prime}(z)} {f^{\prime}(z)}\right)>0 \hspace{0.5cm}\mbox{for}\;\; z\in\mathbb{D}.
\end{align}
\end{defi}

\begin{defi}
Given $\alpha\in[0,1)$, a function $f\in\mathcal{A}$ of the form \eqref{eq-1.1} is called starlike functions of order $\alpha$, if
\begin{align}\label{eq-1.4}
	{\rm Re}\left(\frac{zf^{\prime}(z)}{f(z)}\right)>\alpha \hspace{0.5cm} \mbox{for}\;\; z\in\mathbb{D}.
\end{align}
The class of all such functions is denoted by $\mathcal{S}^*(\alpha)$. For $\alpha=0$, this class reduces to the well-known class $\mathcal{S}^*$, konwn as the class of starlike functions.
\end{defi}

\begin{defi}
Given $\alpha\in[0,1)$, a function $f\in\mathcal{A}$ of the form \eqref{eq-1.1} is called convex functions of order $\alpha$, if
\begin{align}\label{eq-1.5}
	{\rm Re}\left(1+ \frac{zf^{\prime\prime}(z)} {f^{\prime}(z)}\right)>\alpha\hspace{0.5cm}\mbox{for}\;\; z\in\mathbb{D}.
\end{align}
The class of all such functions is denoted by $\mathcal{C}(\alpha)$. For $\alpha=0$, this class reduces to the well-known class $\mathcal{C}$, konwn as the class of convex functions.
\end{defi}

\begin{defi}
A function $f\in\mathcal{A}$ of the form \eqref{eq-1.1} is said to be strongly starlike of order $\alpha$, $0<\alpha\leq 1$, if 
\begin{align}\label{eq-1.6}
	\vline\;\arg\left(\frac{zf^{\prime}(z)}{f(z)}\right)\vline<\frac{\pi\alpha}{2}\hspace{0.5cm}\mbox{for}\;\; z\in\mathbb{D},\;\arg 1:=0.
\end{align}
The class of all such functions is denoted by $\mathcal{S}^*_{\alpha}$.
\end{defi}
\begin{defi}
A function $f\in\mathcal{A}$ of the form \eqref{eq-1.1} is said to be strongly convex of order $\alpha$, $0<\alpha\leq 1$, if 
\begin{align}\label{eq-1.7}
	\vline\;\arg\left(1+\frac{zf^{\prime\prime}(z)}{f^{\prime}(z)}\right)\vline<\frac{\pi\alpha}{2}\hspace{0.5cm}\mbox{for}\;\; z\in\mathbb{D}, \;\arg 1:=0.
\end{align}	
The class of all such functions is denoted by $\mathcal{C}_{\alpha}$.
\end{defi}
It is worth noting that the classes $\mathcal{S}^*_{1}:=\mathcal{S}^*$ and $\mathcal{C}_{1}:=\mathcal{C}$ correspond to the well-known classes of starlike and convex functions, respectively.\vspace{2mm}

Now, we recall a definition of the class \textit{bounded turning functions of order $\alpha$}, which is defined as follows:
\begin{defi}
Let $0\leq\alpha<1$. If $f\in\mathcal{A}$, then $f\in\mathcal{R}(\alpha)$, if and only, if
\begin{align}\label{eq-1.8}
	\mbox{Re} f^{\prime}(z) >\alpha\hspace{0.5cm}\mbox{for}\;\; z\in\mathbb{D}.
\end{align}
It is obvious that $\mathcal{R}(0)=\mathcal{R}$. Elements of the class $\mathcal{R}$ are called \textit{functions of bounded turning}.
\end{defi}\vspace{2mm}

\subsection{Logarithmic coefficients for the class $\mathcal{S}$:}
The \textit{logarithmic coefficients} $\gamma_n$ of $f\in\mathcal{S}$ are defined by
\begin{align}\label{eq-1.9}
	F_{f}(z):=\log\dfrac{f(z)}{z}=2\sum_{n=1}^{\infty}\gamma_{n}(f)z^n, \;\; z\in\mathbb{D},\;\;\log 1:=0.
\end{align}
The coefficients $\gamma_{n} := \gamma_{n}(f)$, $n \in \mathbb{N}$, are called the \textit{logarithmic coefficients} of $f$. Although these coefficients constitute a fundamental structural tool in the theory of univalent functions, establishing inequalities involving them that are sharp remains a challenging task, and exact estimates are known only for a restricted number of function classes. A prime testament to their significance is their central role in Milin's conjecture \cite{Milin-1977-ET} (see also \cite[p. 155]{Duren-1983-NY}), which asserts that for $f \in \mathcal{S}$ and $n \geq 2$,
\begin{equation*}
	\sum_{m=1}^{n}\sum_{k=1}^{m}\left(k|\gamma_{k}|^2 -\frac{1}{k}\right) \leq 0.
\end{equation*}
Equality holds if and only if $f$ is a rotation of the classical Koebe function. The verification of Milin's conjecture by de Branges \cite{Branges-AM-1985} famously yielded the complete proof of the Bieberbach conjecture. 

Beyond this significant result, logarithmic coefficients continue to attract intense scrutiny because obtaining sharp bounds for $|\gamma_n|$ when $n \geq 3$ for the full class $\mathcal{S}$ remains a notorious open problem; explicit sharp bounds are known classically only for $\gamma_{1}$ and $\gamma_{2}$. By differentiating the defining relation \eqref{eq-1.9} and substituting the series expansion \eqref{eq-1.1}, a straightforward calculation yields
\begin{align}\label{eq-1.10}
	\begin{cases}
		\gamma_{1}=\dfrac{1}{2}a_{2},\vspace{1.5mm}\\ \gamma_{2}=\dfrac{1}{2} \left(a_{3} -\dfrac{1}{2}a^2_{2}\right), \vspace{1.5mm}\\ \gamma_{3} =\dfrac{1}{2}\left(a_{4}- a_{2}a_{3} +\dfrac{1}{3}a^3_{2}\right).
	\end{cases}
\end{align}
If $f\in\mathcal{S}$, it is easy to see that  $|\gamma_1|\leq 1$, because $|a_2|\leq 2$. Using the Fekete–Szegö inequality \cite[Theorem 3.8]{Duren-1983-NY} for functions in $\mathcal{S}$ in \eqref{eq-1.3}, we obtain the sharp estimate
\begin{align*}
	|\gamma_{2}|\leq \frac{1}{2}\left(1+2e^{-2}\right) =0.635\ldots.
\end{align*}
For $n\geq 3$, determining bounds for $|\gamma_n|$ when $f\in \mathcal{S}$ appears to be significantly more challenging, and its remains unsolved. \vspace{2mm}

\subsection{Logarithmic inverse coefficients for the class $\mathcal{S}$:}
Let $F$ be the inverse function of $f\in\mathcal{S}$ defined in a neighborhood of the origin with the Taylor series expansion
\begin{align}\label{eq-1.11}
	F(w):=f^{-1}(w)= w+\sum_{n=2}^{\infty} A_n w^n,
\end{align}
where we may choose $|w|<1/4$, as we know that the famous K$\ddot{\mbox{o}}$ebe’s $1/4$-theorem ensures that, for each univalent function $f$ defined in $\mathbb{D}$, it inverse $f^{-1}$ exists at least on a disc of radius $1/4$. There has been a good deal of interest in determining the behavior of the inverse coefficients of $f$ given in \eqref{eq-1.1} when the corresponding function $f$ is restricted to some proper geometric subclasses of $\mathcal{S}$.\vspace{2mm}

Let $f(z)=z+ \sum_{n=2}^{\infty}a_nz^n$ be a function in class $\mathcal{S}$. Since $f(f^{-1}(w))=w$ and using \eqref{eq-1.11} we obtain
\begin{align}\label{eq-1.12}
	\begin{cases}
		A_2= -a_2, \\ A_3=-a_3 +2a^2_{2}, \\ A_4=- a_4 +5a_2 a_3 -5a^3_{2}.
	\end{cases}
\end{align}
The notation of the logarithmic inverse coefficient of $f$ was introduced by Ponnusamy \textit{et al.} (see \cite{Ponnusamy-Sharma-Wirths-RM-2018}). As with $f$, the logarithmic inverse coefficients $\Gamma_n:=\Gamma_n(F)$, $n\in\mathbb{N}$, of $F$ are defined by the equation
\begin{align}\label{eq-1.13}
	F_{f^{-1}}(w):=\log\left(\frac{f^{-1}(w)}{w}\right)=2\sum_{n=1}^{\infty} \Gamma_n(F) w^n \;\;\;\; \mbox{for} \;\;|w|<1/4.
\end{align}
The authors of \cite{Ponnusamy-Sharma-Wirths-RM-2018} obtained the sharp bound for the logarithmic inverse coefficients for the class $\mathcal{S}$ and also for functions for subclasses in $\mathcal{S}$. By differentiating \eqref{eq-1.13} together with \eqref{eq-1.11}, using \eqref{eq-1.12} and then equating coefficients, we obtain
\begin{align}\label{eq-1.14}
	\begin{cases}
		\Gamma_1=-\dfrac{1}{2}a_2, \vspace{2mm}\\ \Gamma_2=-\dfrac{1}{2}\left(a_3 -\frac{3}{2}a^2_{2}\right), \vspace{2mm}\\ \Gamma_3=-\frac{1}{2}\left(a_4 -4a_2 a_3 +\frac{10}{3}a^3_{2}\right).
	\end{cases}
\end{align}

\subsection{Toeplitz determinant with logarithmic coefficients and logarithmic inverse coefficients:}
In recent years, considerable attention has been given to determining the sharp bounds for the second Hankel and Toeplitz determinants involving logarithmic coefficients and logarithmic coefficients of inverse functions of functions \cite{Kowalczyk-Lecko-BAMS-2022, Kowalczyk-Lecko-RACSAM-2023,Allu-Shaji-BAMS.-2024,Mandal-Ahamed-LMJ-2024,Mandal-Ahamed-Zaprawa-MS-2025}. \vspace{2mm}

Hankel and Toeplitz matrices are closely related to each other. Hankel matrices contain constant entries along the reverse diagonals. In $2016$, Ye and Lim \cite{Ye-Lim-FCM-2016} proved that any $n\times n$ matrix over $\mathbb{C}$ generically can be written as the product of some Toeplitz matrices or Hankel matrices. Hankel matrices and their determinants are fundamental in various mathematical disciplines, with numerous applications across different fields \cite{Ye-Lim-FCM-2016}. Similarly, Toeplitz matrices and Toeplitz determinants hold significant importance in both pure and applied mathematics \cite{Toeplitz-1907}. They occur in analysis, integral equations, image processing, signal processing, quantum mechanics and among other areas. For more applications, we refer to the survey article \cite{Ye-Lim-FCM-2016}. \vspace{2mm}

The study of Toeplitz determinants for starlike or many other functions has been done extensively, their sharp bounds have been established. But, the study of Toeplitz determinants, whose entries are logarithmic coefficients, as well as logarithmic coefficients of inverse functions of $f\in\mathcal{S}$ is very little and whether all of those bounds will be sharp or not, is not yet known. In that case, what will be the external functions is also not known till date. The significance of logarithmic coefficients makes the proposed problem worth considering and interesting. The results enlarge the scope of knowledge on logarithmic coefficients. \vspace{2mm}

In this article, we investigate the second Toeplitz determinant associated with logarithmic coefficients and the logarithmic coefficients of the inverse function for functions, as follows,
\begin{align}\label{eq-1.15}
	T_{2,1}(F_{f}/2):=\begin{vmatrix}
		\gamma_{1} & \gamma_{2}\\ \gamma_{2} & \gamma_{1}
	\end{vmatrix}, \hspace{2cm} T_{2,1}(F_{f^{-1}}/2):=\begin{vmatrix}
		\Gamma_{1} & \Gamma_{2}\\ \Gamma_{2} & \Gamma_{1}
	\end{vmatrix}
\end{align}
and
\begin{align}\label{eq-1.16}
	T_{2,2}(F_{f}/2):=\begin{vmatrix}
		\gamma_{2} & \gamma_{3}\\ \gamma_{3} & \gamma_{2}
	\end{vmatrix}, \hspace{2cm} T_{2,2}(F_{f^{-1}}/2):=\begin{vmatrix}
		\Gamma_{2} & \Gamma_{3}\\ \Gamma_{3} & \Gamma_{2}
	\end{vmatrix}.
\end{align}
Despite extensive research on coefficient problems for various classes of analytic functions on the unit disk, the study of logarithmic coefficients, as well as logarithmic inverse coefficients and their formations in the Hankel or Toeplitz determinants, remains largely unexplored, with only limited results available in \cite{Allu-Arora-Shaji-MJM-2023,Allu-Shaji-BAMS-2024,Allu-Shaji-BAMS.-2024,Mandal-Ahamed-LMJ-2024,Man-Roy-Aha-IJS-2024,Kowalczyk-Lecko-BAMS-2022,Kowalczyk-Lecko-BMMS-2022,Kowalczyk-Lecko-RACSAM-2023}. This lack of attention serves as the primary motivation for the present paper. Our main objective is to fill this gap in the existing literature and contribute to the understanding of logarithmic coefficient problems for certain classes of analytic functions, as well as the sharpness of the results. \vspace{2mm}

In this article, we aim to determine sharp bounds for various problems in geometric function theory. These include sharp bounds for the Toeplitz determinants associated with the logarithmic coefficients, as well as logarithmic coefficients of inverse functions of functions. The organization of the paper is as follows: In Section 2, we present preliminary results that play a key role in our proofs. In Section 3, we establish sharp bounds for the Toeplitz determinants with logarithmic coefficients and the logarithmic coefficients of inverse functions for functions in the classes $\mathcal{S}^*(\alpha)$, $\mathcal{C}(\alpha)$, $\mathcal{S}^*_{\alpha}$, $\mathcal{C}_{\alpha}$ and $\mathcal{R}(\alpha)$, respectively. \vspace{2mm}

\section{\bf Preliminary results}
The class $\mathcal{P}$ of all analytic functions $q$ in $\mathbb{D}$ satisfying $q(0)=1$ and $\mbox{Re}\;q(z)>0$ for $z\in\mathbb{D}$. Thus, every $p\in\mathcal{P}$ can be represented as
\begin{align}\label{eq-2.3}
	q(z)=1+\sum_{n=1}^{\infty}b_n z^n,\; z\in\mathbb{D}.
\end{align}
Elements of the class $\mathcal{P}$ are called  Carath$\acute{e}$odory functions. It is known that $|b_n|\leq 2$, $n\geq 1$ for a function $q\in\mathcal{P}$ (see \cite{Duren-1983-NY}). \vspace{2mm}

Suppose $\mathcal{B}_0$ be the class of Schwarz function \textit{i.e} analytic function $w:\mathbb{D}\rightarrow\mathbb{D}$ such that $w(0)=0$. Thus, a function $w\in\mathcal{B}_0$ can be written as a power series
\begin{align}\label{eq-2.1}
	w(z)=\sum_{n=1}^{\infty}c_n z^n.
\end{align}
It is clear that if
\begin{align*}
	p(z)=\frac{1+\omega(z)}{1-\omega(z)}
\end{align*}
then 
\begin{align*}
	p\in\mathcal{P} \;\;\mbox{if and only if}\;\; w\in\mathcal{B}_0.
\end{align*}
Due to the evident connection between Carathéodory functions and Schwarz functions, the results applicable to the coefficients of Schwarz functions are useful in solving coefficient problems associated with starlike functions (see \cite{Zaprawa-RM-2024, Zaprawa-BMMSS-2023}). To prove our main results, we need the following lemmas for the Schwarz functions.
\begin{lem}\cite{Carlson-AMAF-1940}\label{lem-2.1}
Let $ w(z)=c_1 z +c_2 z^2 + c_3 z^3+\ldots $ be a Schwarz function. Then 	
\begin{align*}
	|c_1|\leq 1, \;\;|c_2|\leq 1 -|c_1|^2, \;\;\mbox{and}\;\; |c_3|\leq 1 -|c_1|^2 -\dfrac{|c_2|^2}{1+|c_1|}.
\end{align*}
\end{lem}
\begin{lem}\cite{Prokhorov-Szynal-1981}\label{lem-2.2}
If $ w(z)=c_1 z +c_2 z^2 + c_3 z^3+\ldots\in\mathcal{B}_0 $, then for any real numbers $\mu$ and $\nu$, the following sharp estimate holds:
\begin{align*}
	|c_3 +\mu c_1 c_2 +\nu c^3_1|\leq\Phi(\mu,\nu),
\end{align*}
where
\begin{align*}
	\Phi(\mu,\nu)=\begin{cases}
		1\hspace{2cm}\mbox{if}\hspace{0.5cm}(\mu,\nu)\in D_1\cup D_2\cup \{(2,1)\}, \vspace{2mm}\\ |\nu|\hspace{1.8cm}\mbox{if} \hspace{0.5cm}(\mu,\nu)\in\cup_{k=3}^{7} D_k.
	\end{cases}
\end{align*}
While the sets $D_k, k=1,2,\ldots,7$ are defined as follows:
\begin{align*}
	&D_1=\left\{(\mu,\nu): |\mu|\leq\frac{1}{2},\hspace{0.5cm} -1\leq\nu\leq 1\right\},\vspace{2mm}\\ &D_2=\left\{(\mu,\nu): \frac{1}{2}\leq|\mu|\leq 2,\hspace{0.5cm} \frac{4}{27}(|\mu|+1)^3 -(|\mu|+1)\leq\nu\leq 1\right\},\vspace{2mm}\\&D_3= \left\{(\mu,\nu): |\mu|\leq\frac{1}{2},\hspace{0.5cm} \nu\leq -1\right\},\vspace{2mm}\\ &D_4=\left\{(\mu,\nu): |\mu|\geq \frac{1}{2},\hspace{0.5cm} \nu\leq -\frac{2}{3}(|\mu|+1) \right\},\vspace{2mm}\\ &D_5=\left\{(\mu,\nu): |\mu|\leq 2, \hspace{0.5cm} \nu\geq 1\right\},\vspace{2mm}\\ &D_6= \left\{(\mu,\nu): 2\leq|\mu|\leq 4,\hspace{0.5cm} \nu\geq \frac{1}{12}(\mu^2 +8)\right\},\vspace{2mm}\\ &D_7= \left\{(\mu,\nu): |\mu|\geq 4,\hspace{0.5cm} \nu\leq \frac{2}{3}(|\mu|-1) \right\}.
\end{align*}
\end{lem}

\section{\bf Main Results and Proofs}
We state our first main result, which determines the sharp bound for $|T_{2,1}(F_{f}/2)|$ as well as $|T_{2,2}(F_{f}/2)|$ for functions in the class $\mathcal{S}^*(\alpha)$.
\begin{thm}\label{th-3.1}
Suppose that $0\leq\alpha <1$. If $ f(z)=z+a_2z^2+a_3z^3+\cdots\in \mathcal{S}^*(\alpha)$. Then 
\begin{align*}
	|T_{2,1}(F_{f}/2)|\leq\frac{5(1-\alpha)^2}{4}
\end{align*}
and
\begin{align*}
	|T_{2,2}(F_{f}/2)|\leq\frac{13(1-\alpha)^2}{36}.
\end{align*}
The above inequalities are sharp.
\end{thm}
\begin{proof}
Let $f\in\mathcal{S}^*(\alpha)$ be of the form $ f(z)=z +\sum_{n=2}^{\infty} a_n z^n $, $z\in \mathbb{D}$. By \eqref{eq-1.4}, we have
\begin{align}\label{eq-3.1}
	\frac{zf^{\prime}(z)}{f(z)}=(1-\alpha)\frac{1+\omega(z)}{1-\omega(z)} +\alpha
\end{align}
for some $\omega\in\mathcal{B}_0$ of the form \eqref{eq-2.1}. A routine computation shows that
\begin{align*}
	\frac{zf^{\prime}(z)}{f(z)}= 1+a_2 z +(-a^2_{2} +2a_3)z^2 &+(a^3_{2} -3a_2 a_3 +3a_4)z^3 \\& +(-a^4_{2} +4a^2_{2}a_3 -2a^2_{3} -4a_2 a_4 +4a_5)z^4 +\cdots.
\end{align*}
Also we have
\begin{align*}
	(1-\alpha)\frac{1+\omega(z)}{1-\omega(z)} +\alpha= 1 + 2(1-\alpha) c_1 z &+ 2(1-\alpha)(c^2_1 + c_2)z^2 \\&+ 2(1-\alpha)(c^3_1 + 2c_1 c_2 +c_3)z^3 +\cdots.
\end{align*}
By comparing the coefficients of $z, z^2, z^3, $ from both side of \eqref{eq-3.1}, we get
\begin{align}\label{eq-3.2}
	\begin{cases}
		&a_2= 2(1-\alpha) c_1, \vspace{2mm}\\ &a_3=(1-\alpha) \left((3-2\alpha)c^2_1 +c_2\right), \vspace{2mm}\\ &a_4=\dfrac{2(1-\alpha)}{3}\left((6-7\alpha +2\alpha^2)c^3_1 +(5-3\alpha) c_1 c_2 +c_3\right).
	\end{cases}
\end{align}
In view of \eqref{eq-1.10}, \eqref{eq-3.2} and Lemma \ref{lem-2.1}, we obtain that
\begin{align*}
	|\gamma_{1}|=(1-\alpha)|c_1|\leq (1-\alpha).
\end{align*}
Again, we have
\begin{align*}
	\gamma_{2}=\frac{1}{2}a_{3} -\frac{1}{4}a^2_{2}= \frac{(1-\alpha)}{2}\left(c^2_1 +c_2\right),
\end{align*}
\textit{i.e.},
\begin{align*}
	|\gamma_{2}|\leq \frac{(1-\alpha)}{2}\left(|c_1|^2 +|c_2|\right) \leq \frac{(1-\alpha)}{2}.
\end{align*}
Using \eqref{eq-3.2} and Lemma \ref{lem-2.2}, we obtain
\begin{align*}
	\gamma_{3} =\frac{1}{2}a_{4}- \frac{1}{2}a_{2}a_{3} +\frac{1}{6}a^3_{2}=\frac{(1-\alpha)}{3}\left(c^3_1 +2 c_1 c_2 +c_3\right),
\end{align*}
\textit{i.e.},
\begin{align*}
	|\gamma_{3}|\leq \frac{(1-\alpha)}{3} |c_3 +2 c_1 c_2 +c^3_1| \leq \frac{(1-\alpha)}{3}.
\end{align*}
From \eqref{eq-1.15} and \eqref{eq-1.16}, we determine the following inequalities
\begin{align*}
	|T_{2,1}(F_{f}/2)|\leq|\gamma_{1}|^2 +|\gamma_{2}|^2 \leq \frac{5(1-\alpha)^2}{4}
\end{align*}
and
\begin{align*}
	|T_{2,2}(F_{f}/2)|\leq|\gamma_{2}|^2 +|\gamma_{3}|^2 \leq \frac{13(1-\alpha)^2}{36}.
\end{align*}
Thus, our desired inequalities are established.\vspace{1.2mm}

 To prove the sharpness in the above inequalities, we consider the function $h_1\in\mathcal{S}^*(\alpha)$ such that
\begin{align}\label{Eq-3.3}
	h_1(z)=\frac{z}{(1-iz)^{2(1-\alpha)}}&= z+ 2i(1-\alpha)z^2 -(1-\alpha)(3-2\alpha)z^3 \\&\nonumber\quad -\frac{2}{3}i(2-\alpha) (1-\alpha) (3-2\alpha)z^4 +\cdots.
\end{align}
It is easy to see that 
\begin{align*}
	a_2=2i(1-\alpha),\; a_3=-(1-\alpha)(3-2\alpha)\; \mbox{and}\; a_4=-\frac{2}{3}i(2-\alpha) (1-\alpha) (3-2\alpha).
\end{align*}Also, an elementary computation shows that
\begin{align*}
	|T_{2,1}(F_{f}/2)|=\frac{5(1-\alpha)^2}{4} \hspace{0.5cm}\mbox{and}\hspace{0.5cm} |T_{2,2}(F_{f}/2)|= \frac{13(1-\alpha)^2}{36}.
\end{align*}
This completes the proof.
\end{proof}

We obtain the following immediate corollary for the class $\mathcal{S}^*$.
\begin{cor}
If $ f(z)=z+a_2z^2+a_3z^3+\cdots\in \mathcal{S}^*$, then we have
\begin{align*}
	|T_{2,1}(F_{f}/2)|\leq\frac{5}{4} \hspace{0.5cm}\mbox{and} \hspace{0.5cm} |T_{2,2}(F_{f}/2)|\leq\frac{13}{36}.
\end{align*}
The above inequalities are sharp.
\end{cor}

In the next theorem, we obtain the sharp bound for the Toeplitz determinant $|T_{2,1}(F_{f^{-1}}/2)|$ for the logarithmic coefficients of inverse functions of functions in the class $\mathcal{S}^*(\alpha)$.
\begin{thm}\label{th-3.2}
Assume that $0\leq\alpha <1$. If $ f(z)=z+a_2z^2+a_3z^3+\cdots\in \mathcal{S}^*(\alpha)$. Then 
\begin{align*}
	|T_{2,1}(F_{f^{-1}}/2)|\leq\dfrac{(1-\alpha)^2}{4}(13-24\alpha +16\alpha^2).
\end{align*}
The inequality is sharp.
\end{thm}
\begin{proof}
By utilizing \eqref{eq-1.14} and \eqref{eq-3.2}, we see that
\begin{align*}
	\Gamma_1=-(1-\alpha)c_1 \hspace{0.5cm}\mbox{and}\hspace{0.5cm} \Gamma_2= \frac{(1-\alpha)}{2}\left((3-4\alpha)c^2_1 -c_2\right).
\end{align*}
From \eqref{eq-1.15}, we obtain the following
\begin{align*}
     T_{2,1}(F_{f^{-1}}/2)=\dfrac{(1-\alpha)^2}{4}\left( -(3-4\alpha)^2 c^4_1 -c^2_2 +4c^2_1 +2(3-4\alpha)c^2_1 c_2\right).
\end{align*}
Suppose that $|c_1|=x$. Applying the triangle inequality and Lemma \ref{lem-2.1} yields the following
\begin{align*}
	|T_{2,1}(F_{f^{-1}}/2)|\leq \dfrac{(1- \alpha)^2}{4} \left((3-4\alpha)^2 x^4 +(1-x^2)^2 +4x^2 +2|3-4\alpha|x^2 (1-x^2)\right).
\end{align*}
Now, we consider two cases:\\

\noindent{\bf Case I:} For $0\leq\alpha\leq\frac{3}{4}$, we get
\begin{align*}
	|T_{2,1}(F_{f^{-1}}/2)|&\leq \dfrac{(1- \alpha)^2}{4} \left( 1+8(1-\alpha)x^2 +4(1-2\alpha)^2 x^4\right)\\&:=\dfrac{(1- \alpha)^2}{4}g_1(x).
\end{align*}
\noindent{\bf Case II:} For $\frac{3}{4}<\alpha<1$, we see that
\begin{align*}
	|T_{2,1}(F_{f^{-1}}/2)|&\leq \dfrac{(1- \alpha)^2}{4} \left( 1+4(2\alpha-1)x^2 +16(1-\alpha)^2 x^4\right)\\&:=\dfrac{(1- \alpha)^2}{4}g_2(x).
\end{align*}
In both cases functions $g_j(x)$ for $j=1,2$ are increasing with respect to $x$ in $0\leq x \leq 1$. Hence, we obtain the following
\begin{align*}
	|T_{2,1}(F_{f^{-1}}/2)|&\leq \dfrac{(1-\alpha)^2}{4} g_j(1) \\&= \dfrac{(1-\alpha)^2}{4}(13-24\alpha +16\alpha^2).
\end{align*}
The equalities hold for the function $h_1\in\mathcal{S}^*(\alpha)$ defined in \eqref{Eq-3.3}. It can be easily shown that
\begin{align*}
	|T_{2,1}(F_{f^{-1}}/2)|=\frac{(1-\alpha)^2}{4}(13-24\alpha +16\alpha^2).
\end{align*}
This completes the proof.
\end{proof}

An immediate consequence is the following corollary for the class $\mathcal{S}^*$.
\begin{cor}
If $ f(z)=z+a_2z^2+a_3z^3+\cdots\in \mathcal{S}^*$, then we have 
\begin{align*}
	|T_{2,1}(F_{f^{-1}}/2)|\leq\frac{13}{4}.
\end{align*}
The inequality is sharp.
\end{cor}

In the following result, we establish the sharp bound for $|T_{2,1}(F_{f}/2)|$, as well as $|T_{2,2}(F_{f}/2)|$ for functions in the class $\mathcal{C}(\alpha)$.
\begin{thm}\label{th-3.3}
Assume that $0\leq\alpha <1$. If $ f(z)=z+a_2z^2+a_3z^3+\cdots\in \mathcal{C}(\alpha) $. Then 
\begin{align*}
	|T_{2,1}(F_{f}/2)|\leq\frac{(1-\alpha)^2}{144}(45-6\alpha+\alpha^2)
\end{align*}
and
\begin{align*}
	|T_{2,2}(F_{f}/2)|\leq\frac{(1-\alpha)^2}{144}(13-10\alpha+2\alpha^2).
\end{align*}
All the inequalities are sharp.
\end{thm}
\begin{proof}
Let $f\in\mathcal{C}(\alpha)$ be of the form $ f(z)=z +\sum_{n=2}^{\infty} a_n z^n $, $z\in \mathbb{D}$. By \eqref{eq-1.5}, we have
\begin{align}\label{eq-3.3}
	1+\frac{zf^{\prime\prime}(z)}{f^{\prime}(z)}=(1-\alpha)\frac{1+\omega(z)}{1-\omega(z)} +\alpha
\end{align}
for some $\omega\in\mathcal{B}_0$ of the form \eqref{eq-2.1}. An elementary computation shows that
\begin{align*}
	1+ \frac{zf^{\prime\prime}(z)}{f^{\prime}(z)}=1 +2a_2 z &+(6a_3 -4a^2_{2})z^2 +(12a_4 -18a_2 a_3 +8a^3_{2})z^3 \\&+ (20a_5 -32a_2 a_4 -18a^2_{3} +48a_3 a^2_{2}-16a^4_{2})z^4 +\cdots.
\end{align*}
Moreover, we see that
\begin{align*}
	(1-\alpha)\frac{1+\omega(z)}{1-\omega(z)} +\alpha= 1 + 2(1-\alpha) c_1 z &+ 2(1-\alpha)(c^2_1 + c_2)z^2 \\&+ 2(1-\alpha)(c^3_1 + 2c_1 c_2 +c_3)z^3 +\cdots.
\end{align*}
By comparing the coefficients of powers of $z$ on both sides of \eqref{eq-3.3}, we get
\begin{align}\label{eq-3.4}
	\begin{cases}
		&a_2= (1-\alpha)c_1, \vspace{2mm}\\ &a_3=\dfrac{(1-\alpha)}{3} \left((3-2\alpha)c^2_1 +c_2\right), \vspace{2mm}\\ &a_4= \dfrac{(1-\alpha)}{6}\left((6-7\alpha +2\alpha^2)c^3_1 +(5-3\alpha) c_1 c_2 +c_3\right).
	\end{cases}
\end{align}
Using \eqref{eq-1.10} and \eqref{eq-3.4}, a simple computation shows that
\begin{align*}
	|\gamma_{1}|=\frac{1}{2}(1-\alpha)|c_1|\leq \frac{(1-\alpha)}{2}.
\end{align*}
Again, we have
\begin{align*}
	\gamma_{2}= \frac{(1-\alpha)}{12}\left((3-\alpha)c^2_1 +2c_2\right),
\end{align*}
\textit{i.e.},
\begin{align*}
	|\gamma_{2}|&\leq \frac{(1-\alpha)}{12}\left((3-\alpha)|c_1|^2 +2|c_2|\right)\\&\leq \frac{(1-\alpha)}{12}\left((1-\alpha)|c_1|^2 +2\right) \\& \leq \frac{(1-\alpha)(3-\alpha)}{12}.
\end{align*}
In light of \eqref{eq-3.4} and Lemma \ref{lem-2.2}, we determine
\begin{align*}
	\gamma_{3} =\frac{1}{2}a_{4}- \frac{1}{2}a_{2}a_{3} +\frac{1}{6}a^3_{2}=\frac{(1-\alpha)}{12}\left((2-\alpha)c^3_1 +(3-\alpha) c_1 c_2 +c_3\right)
\end{align*}
\textit{i.e.},
\begin{align*}
	|\gamma_{3}|\leq \frac{(1-\alpha)}{12} |c_3 +(3-\alpha) c_1 c_2 +(2-\alpha)c^3_1| \leq \frac{(1-\alpha)(2-\alpha)}{12}.
\end{align*}
Thus, from \eqref{eq-1.15} and \eqref{eq-1.16}, we establish the following
\begin{align*}
	|T_{2,1}(F_{f}/2)|\leq|\gamma_{1}|^2 +|\gamma_{2}|^2 \leq \frac{(1-\alpha)^2}{144}(45-6\alpha+\alpha^2)
\end{align*}
and
\begin{align*}
	|T_{2,2}(F_{f}/2)|\leq|\gamma_{2}|^2 +|\gamma_{3}|^2 \leq \frac{(1-\alpha)^2}{144}(13-10\alpha+2\alpha^2).
\end{align*}
Thus, the required inequalities are obtained. In order to show the sharpness, we consider the function $h_2\in\mathcal{C}(\alpha)$ of the form
\begin{align*}
h_2(z) &=\frac{1}{2\alpha-1}\left((1-iz)^{2\alpha-1}-1\right) i \\
&= z+i(1-\alpha)z^2-\frac{(1-\alpha)(3-2\alpha)}{3}z^3-i\frac{(1-\alpha)(6-7\alpha +2\alpha^2)}{6}z^4+\ldots
\end{align*}

satisfying
\begin{align}\label{Eq-3.6}
	1+\frac{zh_2^{\prime\prime}(z)}{h_2^{\prime}(z)} =(1-\alpha)\frac{1+iz}{1-iz} +\alpha
\end{align}
for which 
\begin{align*}
	|T_{2,1}(F_{f}/2)|=\frac{(1-\alpha)^2}{144}(45-6\alpha+\alpha^2) 
\end{align*}
and
\begin{align*}
	|T_{2,2}(F_{f}/2)|= \frac{(1-\alpha)^2}{144}(13-10\alpha +2\alpha^2).
\end{align*}
This completes the proof.
\end{proof}

We obtain the following corollary for the class $\mathcal{C}$.
\begin{cor}
If $ f(z)=z+a_2z^2+a_3z^3+\cdots\in \mathcal{C} $, then we have
\begin{align*}
	|T_{2,1}(F_{f}/2)|\leq\frac{5}{16} \hspace{0.5cm}\mbox{and} \hspace{0.5cm} |T_{2,2}(F_{f}/2)|\leq\frac{13}{144}.
\end{align*}
The inequalities are sharp.
\end{cor}

In the next theorem, we prove the sharp bounds for $|T_{2,1}(F_{f^{-1}}/2)|$ associated with the logarithmic coefficients of inverse functions of functions in the class $\mathcal{C}(\alpha)$.
\begin{thm}\label{th-3.4}
Let $0\leq\alpha <1$. If $f(z)=z+a_2z^2+a_3z^3+\cdots\in \mathcal{C}(\alpha)$. Then 
\begin{align*}
	|T_{2,1}(F_{f^{-1}}/2)|\leq \dfrac{5(1-\alpha)^2}{144}(9+6\alpha -5\alpha^2).
\end{align*}
The inequality is sharp.
\end{thm}
\begin{proof}
Considering \eqref{eq-1.14} and \eqref{eq-3.4}, we deduce
\begin{align*}
	\Gamma_1=-\frac{1}{2}(1-\alpha)c_1 \hspace{0.5cm}\mbox{and}\hspace{0.5cm} \Gamma_2= \frac{(1-\alpha)}{12}\left((3-5\alpha)c^2_1 -2c_2\right).
\end{align*}
From \eqref{eq-1.15}, it is easy to see that
\begin{align*}
	T_{2,1}(F_{f^{-1}}/2)=\dfrac{(1-\alpha)^2}{144}\left( -(3-5\alpha)^2 c^4_1 -4c^2_2 +36c^2_1 +4(3-5\alpha)c^2_1 c_2\right).
\end{align*}
Set $|c_1|=x$. By applying the triangle inequality and Lemma \ref{lem-2.1}, it follows that
\begin{align*}
	|T_{2,1}(F_{f^{-1}}/2)|\leq \dfrac{(1- \alpha)^2}{144} \left((3-5\alpha)^2 x^4 +4(1-x^2)^2 +36x^2 +4|3-5\alpha|x^2 (1-x^2)\right).
\end{align*}
Now, we consider two cases:\\
\noindent{\bf Case I:} For $0\leq\alpha\leq\frac{3}{5}$, we have
\begin{align*}
	|T_{2,1}(F_{f^{-1}}/2)|&\leq \dfrac{(1- \alpha)^2}{144} \left( 4+20(2-\alpha)x^2 +(1-5\alpha)^2 x^4\right)\\&:=\dfrac{(1- \alpha)^2}{144}g_3(x).
\end{align*}
\noindent{\bf Case II:} For $\frac{3}{5}<\alpha<1$, we see that
\begin{align*}
	|T_{2,1}(F_{f^{-1}}/2)|&\leq \dfrac{(1- \alpha)^2}{144} \left( 4+4(4+5\alpha)x^2 +25(1-\alpha)^2 x^4\right)\\&:=\dfrac{(1- \alpha)^2}{4}g_4(x).
\end{align*}
In both cases functions $g_j(x)$ for $j=3,4$ are increasing with respect to $x$ in $0\leq x \leq 1$. Hence, the following is derived
\begin{align*}
	|T_{2,1}(F_{f^{-1}}/2)|&\leq \dfrac{(1-\alpha)^2}{144} g_j(1) \\&= \dfrac{5(1-\alpha)^2}{144}(9-6\alpha +5\alpha^2).
\end{align*}
This is the required inequality. To show the sharpness, we employ the function $h_2\in\mathcal{C}(\alpha)$ satisfying the condition \eqref{Eq-3.6}. It can be easily shown that
\begin{align*}
	|T_{2,1}(F_{f^{-1}}/2)|= \frac{5(1-\alpha)^2}{144}(9+6\alpha -5\alpha^2).
\end{align*}
This completes the proof.
\end{proof}

Consequently, we have the following corollary for the class $\mathcal{C}$.
\begin{cor}
If $f(z)=z+a_2z^2+a_3z^3+\cdots\in \mathcal{C}$, then we have
\begin{align*}
	|T_{2,1}(F_{f^{-1}}/2)|\leq\frac{5}{16}.
\end{align*}
The inequality is sharp.
\end{cor}

In the following theorem, we determine the sharp bounds for $|T_{2,1}(F_{f}/2)|$ for functions in the class $\mathcal{S}^*_{\alpha}$.
\begin{thm}\label{th-3.5}
Suppose that $0<\alpha\leq 1$. If $ f(z)=z+a_2z^2+a_3z^3+\cdots\in \mathcal{S}^*_{\alpha}$. Then 
\begin{align*}
	|T_{2,1}(F_{f}/2)|\leq\dfrac{\alpha^2(4+\alpha^2)}{4}.
\end{align*}
The inequality is sharp.
\end{thm}
\begin{proof}
Let $f\in\mathcal{S}^*_{\alpha}$ be of the form $ f(z)=z +\sum_{n=2}^{\infty} a_n z^n $, $z\in \mathbb{D}$. By \eqref{eq-1.6}, we have
\begin{align}\label{eq-3.5}
	\frac{zf^{\prime}(z)}{f(z)}=\left(\frac{1+\omega(z)}{1-\omega(z)}\right)^\alpha
\end{align}
for some $\omega\in\mathcal{B}_0$ of the form \eqref{eq-2.1}. A routine computation shows that
\begin{align*}
	\frac{zf^{\prime}(z)}{f(z)}= 1+a_2 z +(-a^2_{2} +2a_3)z^2 &+(a^3_{2} -3a_2 a_3 +3a_4)z^3 \\& +(-a^4_{2} +4a^2_{2}a_3 -2a^2_{3} -4a_2 a_4 +4a_5)z^4 +\cdots.
\end{align*}
Furthermore, we have
\begin{align*}
	\left(\frac{1+\omega(z)}{1-\omega(z)}\right)^\alpha= 1 + 2\alpha c_1 z &+ 2(\alpha^2 c^2_1 + \alpha c_2)z^2 \\&+ \frac{2}{3}(\alpha c^3_1 +2\alpha^3 c^3_1+ 6\alpha^2 c_1 c_2 +3\alpha c_3)z^3 +\cdots.
\end{align*}
By comparing the coefficients of like powers of $z$ from both side of \eqref{eq-3.5}, we get
\begin{align}\label{eq-3.6}
	\begin{cases}
		&a_2= 2\alpha c_1, \vspace{2mm}\\ &a_3=\alpha \left(3\alpha c^2_1 +c_2\right), \vspace{2mm}\\ &a_4=\dfrac{2\alpha}{9} \left((1+17\alpha^2)c^3_1 +15\alpha c_1 c_2 +3c_3\right).
	\end{cases}
\end{align}
Using \eqref{eq-1.10} and \eqref{eq-3.6}, we obtain
\begin{align*}
	\gamma_{1}=\alpha c_1 \hspace{0.5cm}\mbox{and}\hspace{0.5cm} \gamma_{2}= \frac{\alpha}{2}\left(\alpha c^2_1 +c_2\right).
\end{align*}
From \eqref{eq-1.15}, we have
\begin{align*}
	T_{2,1}(F_{f}/2)=-\dfrac{\alpha^2}{4}\left(\alpha^2 c^4_1 +c^2_2 -4c^2_1 +2\alpha c^2_1 c_2\right).
\end{align*}
Assume that $|c_1|=x$. Using the triangle inequality along with Lemma \ref{lem-2.1}, we proceed as follows
\begin{align*}
	|T_{2,1}(F_{f}/2)|&\leq \dfrac{\alpha^2}{4} \left(\alpha^2 x^4 +(1-x^2)^2 +4x^2 +2\alpha x^2(1-x^2)\right)\\&= \dfrac{\alpha^2}{4} \left(1+ 2(1+\alpha)x^2 +(1-\alpha)^2 x^4\right) \\&:= \dfrac{\alpha^2}{4} g_5(x).
\end{align*}
The function $g_5(x)$ is increasing with respect to $x$ in $0\leq x \leq 1$. Hence, we obtain the following
\begin{align*}
	|T_{2,1}(F_{f}/2)|&\leq \dfrac{\alpha^2}{4} g_5(1) = \dfrac{\alpha^2(4+\alpha^2)}{4}.
\end{align*}
The desired inequality is obtained. To establish the sharpness of the above inequalities, we examine the function $h_3\in\mathcal{S}^*_{\alpha}$ such that
\begin{align}\label{Eq-3.9}
	\frac{zh_3^{\prime}(z)}{h_3(z)}=\left(\frac{1+iz}{1-iz}\right)^\alpha.
\end{align}
for which $a_2=2i\alpha$ and $a_3=-3\alpha^2$. A simple computation shows that 
\begin{align*}
    |T_{2,1}(F_{f}/2)|= \frac{\alpha^2(4+\alpha^2)}{4}.
\end{align*}
This completes the proof.
\end{proof}

We establish the sharp bound for $|T_{2,1}(F_{f^{-1}}/2)|$, which are linked to the logarithmic coefficients of inverse functions of functions in the class $\mathcal{S}^*_{\alpha}$.
\begin{thm}\label{th-3.6}
Let $0<\alpha\leq 1$. If $ f(z)=z+a_2z^2+a_3z^3+\cdots\in \mathcal{S}^*_{\alpha}$. Then 
\begin{align*}
	|T_{2,1}(F_{f^{-1}}/2)|\leq \dfrac{\alpha^2(4 +9\alpha^2)}{4}.
\end{align*}
The inequality is sharp.
\end{thm}
\begin{proof}
In view of \eqref{eq-1.14} and \eqref{eq-3.6}, we obtain
\begin{align*}
	\Gamma_1=-\alpha c_1 \hspace{0.5cm}\mbox{and}\hspace{0.5cm} \Gamma_2= \frac{\alpha}{2}\left(3\alpha c^2_1 -c_2\right).
\end{align*}
From \eqref{eq-1.15}, we see that
\begin{align*}
	T_{2,1}(F_{f^{-1}}/2)=-\dfrac{\alpha^2}{4}\left(9\alpha^2 c^4_1 +c^2_2 -4c^2_1 -6\alpha c^2_1 c_2\right).
\end{align*}
Assume that $|c_1|=x$. Applying the triangle inequality and Lemma \ref{lem-2.1}, we obtain the following
\begin{align*}
	|T_{2,1}(F_{f^{-1}}/2)|&\leq \dfrac{\alpha^2}{4} \left(9\alpha^2 x^4 +(1-x^2)^2 +4x^2 +6\alpha x^2(1-x^2)\right)\\&= \dfrac{\alpha^2}{4} \left(1+ 2(1+3\alpha)x^2 +(1-3\alpha)^2 x^4\right) \\&:= \dfrac{\alpha^2}{4} g_6(x).
\end{align*}
The function $g_6(x)$ is increasing with respect to $x$ in $0\leq x \leq 1$. This yields the following result
\begin{align*}
	|T_{2,1}(F_{f^{-1}}/2)|&\leq \dfrac{\alpha^2}{4} g_6(1) = \dfrac{\alpha^2(4+9\alpha^2)}{4}.
\end{align*}
All the equalities hold for the function $h_3\in \mathcal{S}^*_{\alpha}$, which is defined in \eqref{Eq-3.9}.
Therefore, a simple computation leads to
\begin{align*}
   |T_{2,1}(F_{f^{-1}}/2)|= \frac{\alpha^2(4 +9\alpha^2)}{4}. 
\end{align*}
This completes the proof.
\end{proof}

The following theorem establishes the best possible bounds for $|T_{2,1}(F_{f}/2)|$ for functions belonging to $\mathcal{C}_{\alpha}$.
\begin{thm}\label{th-3.7}
Assume that $0<\alpha\leq 1$. If $ f(z)=z+a_2z^2+a_3z^3+\cdots\in \mathcal{C}_{\alpha} $. Then 
\begin{align*}
	|T_{2,1}(F_{f}/2)|\leq\dfrac{\alpha^2(4+\alpha^2)}{16}.
\end{align*}
The inequality is sharp.
\end{thm}
\begin{proof}
Let $f\in\mathcal{C}_{\alpha}$ be of the form $ f(z)=z +\sum_{n=2}^{\infty} a_n z^n $, $z\in \mathbb{D}$. By \eqref{eq-1.7}, we have
\begin{align}\label{eq-3.7}
	1+\frac{zf^{\prime\prime}(z)}{f^{\prime}(z)}=\left(\frac{1+\omega(z)}{1-\omega(z)}\right)^\alpha
\end{align}
for some $\omega\in\mathcal{B}_0$ of the form \eqref{eq-2.1}. An elementary computation shows that
\begin{align*}
	1+ \frac{zf^{\prime\prime}(z)}{f^{\prime}(z)}=1 +2a_2 z &+(6a_3 -4a^2_{2})z^2 +(12a_4 -18a_2 a_3 +8a^3_{2})z^3 \\&+ (20a_5 -32a_2 a_4 -18a^2_{3} +48a_3 a^2_{2}-16a^4_{2})z^4 +\cdots.
\end{align*}
In addition, we have
\begin{align*}
	\left(\frac{1+\omega(z)}{1-\omega(z)}\right)^\alpha= 1 + 2\alpha c_1 z &+ 2(\alpha^2 c^2_1 + \alpha c_2)z^2 \\&+ \frac{2}{3}(\alpha c^3_1 +2\alpha^3 c^3_1+ 6\alpha^2 c_1 c_2 +3\alpha c_3)z^3 +\cdots.
\end{align*}
Equating the coefficients of corresponding powers of $z$ on both sides of \eqref{eq-3.7}, we get
\begin{align}\label{eq-3.8}
	\begin{cases}
		&a_2=\alpha c_1, \vspace{2mm}\\ &a_3=\dfrac{\alpha}{3} \left(3\alpha c^2_1 +c_2\right), \vspace{2mm}\\ &a_4= \dfrac{\alpha}{18}\left((1+17\alpha^2)c^3_1 +15\alpha c_1 c_2 +3c_3\right).
	\end{cases}
\end{align}
Using \eqref{eq-1.10} and \eqref{eq-3.8}, a simple computation shows that
\begin{align*}
	\gamma_{1}=\frac{\alpha}{2}c_1 \hspace{0.5cm}\mbox{and}\hspace{0.5cm}
	\gamma_{2}= \frac{\alpha}{12}\left(3\alpha c^2_1 +2c_2\right).
\end{align*}
From \eqref{eq-1.15}, a simple computation shows that
\begin{align*}
	T_{2,1}(F_{f}/2)=-\dfrac{\alpha^2}{144}\left(9\alpha^2 c^4_1 +4c^2_2 -36c^2_1 +12\alpha c^2_1 c_2\right).
\end{align*}
Suppose that $|c_1|=x$, and using the triangle inequality together with Lemma \ref{lem-2.1} yields the following
\begin{align*}
	|T_{2,1}(F_{f}/2)|&\leq \dfrac{\alpha^2}{144} \left(9\alpha^2 x^4 +4(1-x^2)^2 +36x^2 +12\alpha x^2(1-x^2)\right)\\&= \dfrac{\alpha^2}{144} \left(4+ 4(7+3\alpha)x^2 +(2-3\alpha)^2 x^4\right) \\&:= \dfrac{\alpha^2}{144} g_7(x).
\end{align*}
The function $g_7(x)$ is increasing with respect to $x$ in $0\leq x \leq 1$. Hence, we obtain the following
\begin{align*}
	|T_{2,1}(F_{f}/2)|&\leq \dfrac{\alpha^2}{144} g_7(1) = \dfrac{\alpha^2(4+\alpha^2)}{16}.
\end{align*}
To confirm the sharpness, we examine the function $h_4\in\mathcal{C}_{\alpha}$ satisfying
\begin{align}\label{Eq-3.11}
	1+\frac{zh_4^{\prime\prime}(z)}{h_4^{\prime}(z)}=\left(\frac{1+iz}{1-iz}\right)^\alpha
\end{align}
for which $a_2=i\alpha$ and $a_3=-\alpha^2$. Therefore, we see that
\begin{align*}
	 |T_{2,1}(F_{f}/2)|= \frac{\alpha^2(4+\alpha^2)}{16}. 
\end{align*}
This completes the proof.
\end{proof}

The next theorem determines the precise bound for $|T_{2,1}(F_{f^{-1}}/2)|$ associated with the logarithmic coefficients of inverse functions of functions in $\mathcal{C}_{\alpha}$.
\begin{thm}\label{th-3.8}
Suppose that $0<\alpha\leq 1$. If $f(z)=z+a_2z^2+a_3z^3+\cdots\in \mathcal{C}_{\alpha}$. Then 
\begin{align*}
	|T_{2,1}(F_{f^{-1}}/2)|\leq\dfrac{\alpha^2(4+\alpha^2)}{16}.
\end{align*}
The inequality is sharp.
\end{thm}
\begin{proof}
With reference to \eqref{eq-1.14} and \eqref{eq-3.8}, we see that
\begin{align*}
	\Gamma_1=-\frac{\alpha}{2}c_1 \hspace{0.5cm}\mbox{and}\hspace{0.5cm}
	\Gamma_2= \frac{\alpha}{12}\left(3\alpha c^2_1 -2c_2\right).
\end{align*}
From \eqref{eq-1.15}, we see that
\begin{align*}
	T_{2,1}(F_{f^{-1}}/2)=\dfrac{\alpha^2}{144}\left(-9\alpha^2 c^4_1 -4c^2_2 +36c^2_1 +12\alpha c^2_1 c_2\right).
\end{align*}
Set $|c_1|=x$. By applying the triangle inequality and Lemma \ref{lem-2.1}, it follows that 
\begin{align*}
	|T_{2,1}(F_{f^{-1}}/2)|&\leq \dfrac{\alpha^2}{144} \left(9\alpha^2 x^4 +4(1-x^2)^2 +36x^2 +12\alpha x^2(1-x^2)\right)\\&= \dfrac{\alpha^2}{144} \left(4+ 4(7+3\alpha)x^2 +(2-3\alpha)^2 x^4\right) \\&:= \dfrac{\alpha^2}{144} g_8(x).
\end{align*}
The function $g_8(x)$ is increasing with respect to $x$ in $0\leq x \leq 1$. Hence, we obtain the desired inequality
\begin{align*}
	|T_{2,1}(F_{f^{-1}}/2)|&\leq \dfrac{\alpha^2}{144} g_8(1) = \dfrac{\alpha^2(4+\alpha^2)}{16}.
\end{align*}
To illustrate the sharpness, we consider the function $h_4\in\mathcal{C}_{\alpha}$ satisfying the condition \eqref{Eq-3.11} and obtain the following
\begin{align*}
   |T_{2,1}(F_{f^{-1}}/2)|=\frac{\alpha^2(4+\alpha^2)}{16}.
\end{align*}
This completes the proof.
\end{proof}

We now present a result that provides the sharp bound for $|T_{2,1}(F_{f}/2)|$ for functions in the class $\mathcal{R}(\alpha)$.
\begin{thm}\label{th-3.9}
Assume that $0\leq\alpha <1$. If $ f(z)=z+a_2z^2+a_3z^3+\cdots\in \mathcal{R}(\alpha) $. Then 
\begin{align*}
	|T_{2,1}(F_{f}/2)|\leq\frac{(1-\alpha)^2}{144}(37 + 6\alpha +9\alpha^2).
\end{align*}
The inequality is sharp.
\end{thm}
\begin{proof}
Let $f\in\mathcal{R}(\alpha)$ be of the form $ f(z)=z +\sum_{n=2}^{\infty} a_n z^n $, $z\in \mathbb{D}$. By \eqref{eq-1.8}, we have
\begin{align}\label{eq-3.9}
	f^{\prime}(z)=(1-\alpha)\frac{1+\omega(z)}{1-\omega(z)} +\alpha
\end{align}
for some $\omega\in\mathcal{B}_0$ of the form \eqref{eq-2.1}. An elementary computation shows that
\begin{align*}
	f^{\prime}(z)=1 +2a_2 z + 3a_3z^2 + 4a_4z^3+5a_5z^4 +\cdots.
\end{align*}
We also derive
\begin{align*}
	(1-\alpha)\frac{1+\omega(z)}{1-\omega(z)} +\alpha= 1 + 2(1-\alpha) c_1 z &+ 2(1-\alpha)(c^2_1 + c_2)z^2 \\&+ 2(1-\alpha)(c^3_1 + 2c_1 c_2 +c_3)z^3 +\cdots.
\end{align*}
By comparing the coefficients of powers of $z$ on both sides of \eqref{eq-3.9}, we get
\begin{align}\label{eq-3.10}
	\begin{cases}
		&a_2= (1-\alpha)c_1, \vspace{2mm}\\ &a_3=\dfrac{2(1- \alpha)}{3} \left(c^2_1 +c_2\right), \vspace{2mm}\\ &a_4= \dfrac{(1-\alpha)}{2}\left(c^3_1 +2c_1 c_2 +c_3\right).
	\end{cases}
\end{align}
Using \eqref{eq-1.10} and \eqref{eq-3.10}, we get
\begin{align*}
	\gamma_{1}=\frac{1}{2}(1-\alpha)c_1 \hspace{0.5cm}\mbox{and}\hspace{0.5cm} \gamma_{2}= \frac{(1-\alpha)}{12}\left((3\alpha+1)c^2_1 +4c_2\right).
\end{align*}
From \eqref{eq-1.15}, 
\begin{align*}
	T_{2,1}(F_{f}/2) = -\frac{1}{144}(1-\alpha)^2\left((3\alpha+1)^2c_1^4+8(3\alpha+1)c_1^2c_2+16c_2^2-36c_1^2\right).
\end{align*}
Let $|c_1|=x$. By applying the triangle inequality and Lemma \ref{lem-2.1}, it follows that 
\begin{align*}
	|T_{2,1}(F_{f}/2)|&\leq \frac{1}{144}(1-\alpha)^2 \left((3\alpha+1)^2x^4+8(3\alpha+1)x^2(1-x^2)+16(1-x^2)^2+36x^2\right)\\&= \frac{1}{144}(1-\alpha)^2 \left(16+12(2\alpha+1)x^2+9(\alpha-1)^2x^4\right) \\&:= \dfrac{\alpha^2}{144} g_9(x).
\end{align*}
The function $g_9(x)$ is increasing with respect to $x$ in $0\leq x \leq 1$. Hence, we obtain the sharp desired inequality
\begin{align*}
	|T_{2,1}(F_{f}/2)|&\leq \dfrac{(1-\alpha)^2}{144} g_9(1) \\&= \frac{(1-\alpha)^2}{144}(37 + 6\alpha +9\alpha^2).
\end{align*}
To establish the sharp bound, we consider the function $h_5\in\mathcal{R}(\alpha)$ satisfying
\begin{align}\label{Eq-3.14}
	h_5^{\prime}(z) =(1-\alpha)\frac{1+iz}{1-iz} +\alpha
\end{align}
for which 
\begin{align*}
	a_2=i(1-\alpha)\; \mbox{and}\; a_3=-\frac{2(1-\alpha)}{3}.
\end{align*} As a consequence of these calculations, we arrive at the equality
\begin{align*}
	|T_{2,1}(F_{f}/2)|=\frac{(1-\alpha)^2}{144}(37 + 6\alpha +9\alpha^2).
\end{align*}
This completes the proof.
\end{proof}

From this, we obtain the following corollary for the class $\mathcal{R}$.
\begin{cor}
If $ f(z)=z+a_2z^2+a_3z^3+\cdots\in \mathcal{R} $, then we obtain
\begin{align*}
	|T_{2,1}(F_{f}/2)|\leq\frac{13}{36}.
\end{align*}
The inequality is sharp.
\end{cor}

In the next result, we determine the best possible bound for $|T_{2,1}(F_{f^{-1}}/2)|$ in relation to the logarithmic coefficients of inverse functions of functions in $\mathcal{R}(\alpha)$.
\begin{thm}\label{th-3.10}
Let $0\leq\alpha <1$. If $f(z)=z+a_2z^2+a_3z^3+\cdots\in \mathcal{R}(\alpha)$. Then 
\begin{align*}
	|T_{2,1}(F_{f^{-1}}/2)|\leq\dfrac{(1-\alpha)^2}{144}(61- 90\alpha +81\alpha^2).
\end{align*}
The inequality is sharp.
\end{thm}
\begin{proof}
In view of \eqref{eq-1.14} and \eqref{eq-3.10}, we have
\begin{align*}
	\Gamma_1=-\frac{1}{2}(1-\alpha)c_1 \hspace{0.5cm}\mbox{and}\hspace{0.5cm} \Gamma_2= \frac{(1-\alpha)}{12}\left((9\alpha-5)c^2_1 -4c_2\right).
\end{align*}
From \eqref{eq-1.15}, we obtain the following
\begin{align*}
	T_{2,1}(F_{f^{-1}}/2)=\dfrac{(1-\alpha)^2}{144}\left( -(9\alpha -5)^2 c^4_1 -16c^2_2 +36c^2_1 +8(9\alpha- 5)c^2_1 c_2\right).
\end{align*}
Assume that $|c_1|=x$. Applying the triangle inequality and Lemma \ref{lem-2.1}, we obtain the following
\begin{align*}
	|T_{2,1}(F_{f^{-1}}/2)|\leq \dfrac{(1- \alpha)^2}{144} \left((9\alpha- 5)^2 x^4 +16(1-x^2)^2 +36x^2 +8|9\alpha- 5|x^2 (1-x^2)\right).
\end{align*}
To complete the proof, we consider two cases:\\

\noindent{\bf Case I:} For $0\leq\alpha<\frac{5}{9}$, we obtain
\begin{align*}
	|T_{2,1}(F_{f^{-1}}/2)|&\leq \dfrac{(1- \alpha)^2}{144} \left( 16+(44-72\alpha)x^2 +(1-9\alpha)^2 x^4\right)\\&:=\dfrac{(1- \alpha)^2}{144}g_{10}(x).
\end{align*}
\noindent{\bf Case II:} For $\frac{5}{9}\leq\alpha<1$, we obtain
\begin{align*}
	|T_{2,1}(F_{f^{-1}}/2)|&\leq \dfrac{(1- \alpha)^2}{144} \left( 16+36(2\alpha-1)x^2 +81(1-\alpha)^2 x^4\right)\\&:=\dfrac{(1- \alpha)^2}{144}g_{11}(x).
\end{align*}
In both cases functions $g_j(x)$ for $j=10,11$ are increasing with respect to $x$ in $0\leq x \leq 1$. Hence, we obtain the following
\begin{align*}
	|T_{2,1}(F_{f^{-1}}/2)|&\leq \dfrac{(1-\alpha)^2}{144} g_j(1) \\&= \dfrac{(1-\alpha)^2}{144}(61-90\alpha +81\alpha^2).
\end{align*}
To illustrate the sharpness, we consider the function $h_5\in\mathcal{R}(\alpha)$ satisfying the condition \eqref{Eq-3.14} and obtain the following
\begin{align*}
	|T_{2,1}(F_{f^{-1}}/2)|=\frac{(1-\alpha)^2}{144}(61- 90\alpha +81\alpha^2).
\end{align*}
This completes the proof.
\end{proof}

The next result is an immediate corollary of the above for the class $\mathcal{R}$.
\begin{cor}
If $f(z)=z+a_2z^2+a_3z^3+\cdots\in \mathcal{R}$, then we have
\begin{align*}
	|T_{2,1}(F_{f^{-1}}/2)|\leq\frac{61}{144}.
\end{align*}
The inequality is sharp.
\end{cor}

\noindent{\bf Acknowledgment:} The authors would like to thank the referee(s) for their helpful suggestions and comments for the improvement of the exposition of the paper.\\

\noindent{\bf Author Contributions:} All authors actively worked on the research contained in the paper. All authors reviewed the manuscript.\\

\noindent\textbf{Compliance of Ethical Standards:}\\

\noindent\textbf{Conflict of interest.} The authors declare that there is no conflict  of interest regarding the publication of this paper.\vspace{1.5mm}

\noindent\textbf{Data availability statement.}  Data sharing is not applicable to this article as no datasets were generated or analyzed during the current study.


\begin{thebibliography}{99}	 
	
	\bibitem{Allu-Arora-Shaji-MJM-2023} {\sc V. Allu, V. Arora}, and {\sc A. Shaji}, On the Second Hankel Determinant of Logarithmic Coefficients for Certain Univalent Functions, \textit{Mediterr. J. Math.} \textbf{20} (2023), 81.
	
	\bibitem{Allu-Shaji-BAMS-2024} {\sc V. Allu} and {\sc A. Shaji}, Second Hankel determinant for logarithmic inverse coefficients of convex and starlike functions, \textit{Bull. Aust. Math. Soc.} (2024), 1-12. 
	
	\bibitem{Allu-Shaji-BAMS.-2024} {\sc V. Allu} and {\sc A. Shaji}, The sharp bound of the second Hankel determinant of logarithmic coefficients for starlike and convex functions, \textit{Bull. Aust. Math. Soc.} (2024), 1-9.
	
	\bibitem{Allu-Thomas-2018} {\sc V. Allu} and {\sc D. K. Thomas}, The logarithmic coefficients of univalent functions-an overview, \textit{Current Research in Mathematical and Computer Sciences II Publisher UWM}, Olsztyn (2018), pp. 257–269.
	
	\bibitem{Branges-AM-1985} {\sc L. de Branges}, A proof of the Bieberbach conjecture, \textit{Acta Math.} \textbf{154} (1985), 137-152.
	
	\bibitem{Carlson-AMAF-1940} {\sc F. Carlson}, Sur les coefficients d’une fonction bornée dans le cercle unité, \textit{Ark. Mat. Astr. Fys.} \textbf{27A}(1), 8 (1940).
	
	\bibitem{Duren-1983-NY}{\sc P. T. Duren}, Univalent Functions. \textit{Springer-Verlag}, New York Inc (1983).
	
	\bibitem{Goodman-1983} {\sc A. W. Goodman}, Univalent Functions (Mariner, Tampa, FL, 1983).
	
	\bibitem{Kowalczyk-Lecko-BAMS-2022} {\sc B. Kowalczyk} and {\sc A. Lecko}, Second Hankel determinant of logarithmic coefficients of convex and starlike functions, \textit{Bull. Aust. Math. Soc.} \textbf{105} (2022), no. 3, 458–467.
	
	\bibitem{Kowalczyk-Lecko-BMMS-2022} {\sc B. Kowalczyk} and {\sc A. Lecko}, Second Hankel determinant of logarithmic coefficients of convex and starlike functions of order alpha, \textit{Bull. Malays. Math. Sci. Soc.} \textbf{45} (2022), no. 2, 727–740.
	
	\bibitem{Kowalczyk-Lecko-RACSAM-2023} {\sc B. Kowalczyk} and {\sc A. Lecko}, The second Hankel determinant of the logarithmic coefficients of strongly starlike and strongly convex functions,\textit{ Rev. Real Acad. Cienc. Exactas Fis. Nat. Ser. A-Mat.} \textbf{117}, 91 (2023). 
	
	\bibitem{Mandal-Ahamed-LMJ-2024} {\sc S. Mandal} and {\sc M.B. Ahamed}, Second Hankel determinant of logarithmic coefficients of inverse functions in certain classes of univalent functions, \textit{Lith. Math. J.}  \textbf{64} (2024), 67–79.
	
	\bibitem{Mandal-Ahamed-Zaprawa-MS-2025} {\sc S. Mandal, M. B. Ahamed} and {\sc P. Zaprawa}, Sharp bounds on the logarithmic coefficients of inverse functions for certain classes of univalent functions, \textit{Math. Slovaca} \textbf{75}(5) (2025), 1119-1134.
	
	\bibitem{Man-Roy-Aha-IJS-2024} {\sc S. Mandal, P. P. Roy} and {\sc M. B. Ahamed}, Hankel and Toeplitz determinants of logarithmic coefficients of Inverse functions for certain classes of univalent functions, \textit{Iran. J. Sci.} \textbf{49}, (2024) 243–252.
	
	\bibitem{Milin-1977-ET}{\sc I. M. Milin}, Univalent Functions and Orthonormal Systems (Nauka, Moscow, 1971) (in Russian); English translation, Translations of Mathematical Monographs, 49 (\textit{American Mathematical Society}, Providence, RI, 1977).
	
	\bibitem{Ponnusamy-Sharma-Wirths-RM-2018} {\sc S. Ponnusamy, N. L. Sharma} and {\sc KJ. Wirths}, Logarithmic Coefficients of the Inverse of Univalent Functions, \textit{Results Math} \textbf{73}, 160 (2018).
	
	\bibitem{Prokhorov-Szynal-1981} {\sc D. V. Prokhorov} and {\sc J. Szynal}, Inverse coefficients for ($\alpha$,$\beta$) convex functions, \textit{Ann. Univ. Mariae Curie-Sklodowska Sect. A} \textbf{35} (1981), 125–143.
	
	\bibitem{Sim-Thomas-S-2020} {\sc Y. J. Sim} and {\sc D. K. Thomas}, On the difference of inverse coefficients of univalent functions, \textit{Symmetry} \textbf{12}(12) (2020).
	
	\bibitem{Sumer-Lecko-BMMSS-2023} {\sc S. E. S\"umer, A. Lecko, B. Çekiç} and {\sc B. Seker}, The Second Hankel Determinant of Logarithmic Coefficients for Strongly Ozaki Close-to-Convex Functions, \textit{Bull. Malays. Math. Sci. Soc.} \textbf{46} (2023), 183.
	
	\bibitem{Toeplitz-1907} {\sc O. Toeplitz}, Zur Transformation der Scharen bilinearer Formen von unendlichvielen Ver\"anderlichen. Mathematischphysikalis- che, Klasse, Nachr. der Kgl. Gessellschaft derWissenschaften zu G\"ottingen, (1907) 110–115.
	 
	\bibitem{Ye-Lim-FCM-2016} {\sc K. Ye} and {\sc L. H. Lim}, Every matrix is a product of Toeplitz matrices,\textit{ Found. Comput. Math.} \textbf{16} (2016), no. 3, 577–598.
	
	\bibitem{Zaprawa-BMMSS-2023} {\sc P. Zaprawa}, Inequalities for the Coefficients of Schwarz Functions, \textit{Bull. Malays. Math. Sci. Soc.} 46:144, (2023).
	
	\bibitem{Zaprawa-RM-2024} {\sc P. Zaprawa}, On a Coefficient Inequality for Carathéodory Functions, \textit{Results Math.} (2024) 79:30.
\end{thebibliography}
\end{document}